%% file: main.tex
\documentclass[3p]{elsarticle}

\makeatletter
\def\ps@pprintTitle{%
 \let\@oddhead\@empty
 \let\@evenhead\@empty
 \def\@oddfoot{\centerline{\thepage}}%
 \let\@evenfoot\@oddfoot}
\makeatother

\usepackage{lineno,hyperref}
\modulolinenumbers[5]

\journal{}

\usepackage{subcaption}
\usepackage[ruled,vlined,linesnumbered]{algorithm2e}
\usepackage{graphicx}
\usepackage{amsmath,amsthm,bm}
\usepackage{xcolor}
\definecolor{red}{rgb}{0.9,0,0}
\definecolor{green}{rgb}{0.2,0.6,0.3}
\definecolor{blue}{rgb}{0.,0.3,1}
\definecolor{brown}{rgb}{0.47,0.27,0.14}
\usepackage{float}
\usepackage{natbib}

\input{commondefinitions}

\def\bs#1{\boldsymbol{#1}}
\newtheorem{lemma}{Lemma}

\newcommand{\achange}{}

\newcommand{\cchange}{}

\biboptions{authoryear}

\begin{document}

\begin{frontmatter}

\title{A Probabilistic Framework and a Homotopy Method for Real-time Hierarchical Freight Dispatch Decisions} 

%

\author{Roozbeh Yousefzadeh*}
\cortext[mycorrespondingauthor]{Corresponding author}
\address{Computer Science Department \\
Yale University, New Haven, CT, USA \\
roozbeh.yousefzadeh@yale.edu}

\author{Dianne P. O'Leary} 
\address{Computer Science Department and Institute for Advanced Computer Studies \\
University of Maryland, College Park, MD, USA \\
oleary@cs.umd.edu}




\begin{abstract}
We propose a real-time decision framework for multimodal freight dispatch through a system of hierarchical hubs, using a probabilistic model for transit times. Instead of assigning a fixed time to each transit, we advocate using historical records to identify characteristics of the probability density function for each transit time. We formulate a nonlinear optimization problem that defines dispatch decisions that minimize expected cost, using this probabilistic information. Finally, we propose an effective homotopy algorithm that (empirically) outperforms standard optimization algorithms on this problem by taking advantage of its structure, and we demonstrate its effectiveness on numerical examples.
\end{abstract}

\begin{keyword}
Probabilistic transfer times, Multimodal freight transport, Real-time decisions, Homotopy optimization, Hierarchical systems
\MSC[2010] 90B06 \sep 65H20
\end{keyword}

\end{frontmatter}


\section{Introduction} \label{intro}

In operating a multi-level freight transportation system, the objective is to make dispatch decisions that \achange{minimize the total cost in the system}. Given a fleet of vehicles, freight locations, transit times (including loading/unloading), rewards for on-time delivery, and penalties for late delivery, we decide dispatch times for the freight \achange{vehicles in real-time}. 

\achange{
The problem of deciding dispatch times in a transportation system can be studied in two contexts. The first context is the \textit{planning}, where a pre-planned schedule is decided for the system. The second context is to continuously adjust the dispatch decisions based on \cchange{the} \textit{real-time} state of the system. We study the second context.}

Several previous studies have investigated dispatch decisions in transportation operations as reviewed by \cite{steadieseifi2014multimodal}. Most of the literature assumes fixed known (deterministic) transit times \achange{in freight systems, for both the planning and real-time decisions,} but clearly this is not a realistic assumption. Examination of recent records for any system will typically reveal variations in transit time. Some of these variations follow regular patterns depending, for example, on time of day or season of the year. Other variations are less predictable, due, perhaps, to unusual traffic, weather, or staffing levels. \achange{In general, it is reasonable to assume that the completion time of activities in a system can be quite different than what has been envisioned during the planning.}

\achange{
Furthermore, in a complex network, many of the activities are intertwined}. \cchange{Even when the pre-planned schedule is  based on probabilistic transfer times, the schedule can easily become suboptimal or even useless due to small changes in the completion} \achange{time of just a few activities. Therefore, it is important to continuously adjust the dispatch decisions, according to the situation in real-time. }

\cchange{
\cite{sun2016holding}} \achange{proposed a vehicle holding method for mitigating the effect of service disruptions on coordinated intermodal freight operation. They consider probabilistic transfer times, but their focus is on the impact of correlations among vehicle arrivals to a hub and their system is not hierarchical. They show that the expected value of total cost at a hub is not affected by correlated arrivals.

\cchange{A few} other studies that incorporate probabilistic completion time of activities in freight systems have considered simple systems, not hierarchical ones. For example, \cite{li2016real} studied real-time schedule recovery policies for an individual ship with probabilistic transfer times. However, their approach is more focused on generating policies, and their network is simple, representing the route that a ship follows from one hub to another. Other studies, such as \cite{bock2010real}, consider intermodal networks with multiple hubs, but do not consider probabilistic transfer times.

There are also studies that consider probabilistic transfer times in public passenger systems, in order to adjust the dispatch times in real-time, for example, \cite{berrebi2015real}, \cite{sanchez2016real}, \cite{wu2016designing}, and \cite{wu2019stochastic}. However, the differences between a public passenger system and a freight system are significant i.e., transportation mode, size and hierarchy of network, objective and optimization method. 
}

Our work is motivated by the existence of networks of shipping hubs. \achange{ Such networks may be operated by shipping companies such as UPS and FedEx Express \citep{bowen2012spatial}, global cargo companies \citep{rodrigue2016geography,lorange2005shipping} such as Maersk \citep{fremont2007global}, or domestic freight companies such as Amazon and USPS \citep{dobbins2007overview}. All these networks are hierarchical and their complexities would not be accurately modeled by considering individual hubs.} 

The economic cost of shipping delays to the shipper is well accepted \citep{gong2012assessing}, so our goal is to provide a model and a solution algorithm that can model shipping costs better than deterministic models\achange{ in the context of a hierarchical system.}

We consider a freight transportation system that is multimodal and consists of transfer terminals (hubs) and links connecting the origins to hubs and eventually to destinations. The activities relevant to the operation of such systems consist of transporting the cargo among the hubs and sorting/re-distributing the cargo within the hubs.
There are four main innovations in our work:
\begin{enumerate}
    \item We propose using a probabilistic model for transit times \achange{in a hierarchical system}. Instead of assigning a fixed time to each transit, we advocate using historical records to identify characteristics of the probability density function for each transit time. (Section \ref{sect-model})
    \item \achange{We advocate for a realistic hierarchical model that incorporates the complexity of interrelated activities in multimodal freight systems. Probabilistic transfer times has not been considered for such networks. (Section \ref{sect-model})}
    \item We formulate an optimization problem that determines dispatch decisions that minimize expected cost, using this probabilistic information . \achange{Our formulation is efficient and parallelizable so that it can be solved for real-time purposes.} (Section \ref{sect-problem})
    \item Finally, we propose an effective homotopy algorithm that (empirically) outperforms standard optimization algorithms on this problem by taking advantage of its structure. (Section \ref{sect-alg})
\end{enumerate}
In Section \ref{sect-extention}, we show how our work applies to bidirectional networks and to networks with mixed hierarchies. Section \ref{sect-results} contains illustrative numerical results obtained from the model and the homotopy algorithm. We investigate the effectiveness of the model, measure the performance of algorithm and compare it to well-known global optimization algorithms.

\section{Our probabilistic model} \label{sect-model}

In this section, we consider a unidirectional network, as shown in the example in Figure \ref{fig:network}, and later in Section \ref{sect-extention} we investigate bidirectional systems. In Figure \ref{fig:network}, there are two levels of hubs inter-connected with routes; there are no limits on the number of spokes into or out of a hub. The cargo flow is consistent with the direction of arrows in the figure, from in-bound routes to the $1^{st}$ level hub, then to the $2^{nd}$ level, and eventually to delivery routes. Each of the routes has a pre-planned schedule and headway. The model presented in this paper is based on a network with two levels of hubs but the formulation can be extended to consider more levels. The problem is defined as the situation where disruptions have occurred in the system and consequently, the schedule of dispatches in the system should be re-optimized based on the real-time information about operations in the system.

\begin{figure}[h!]
\centering
\includegraphics[width=0.75\textwidth]{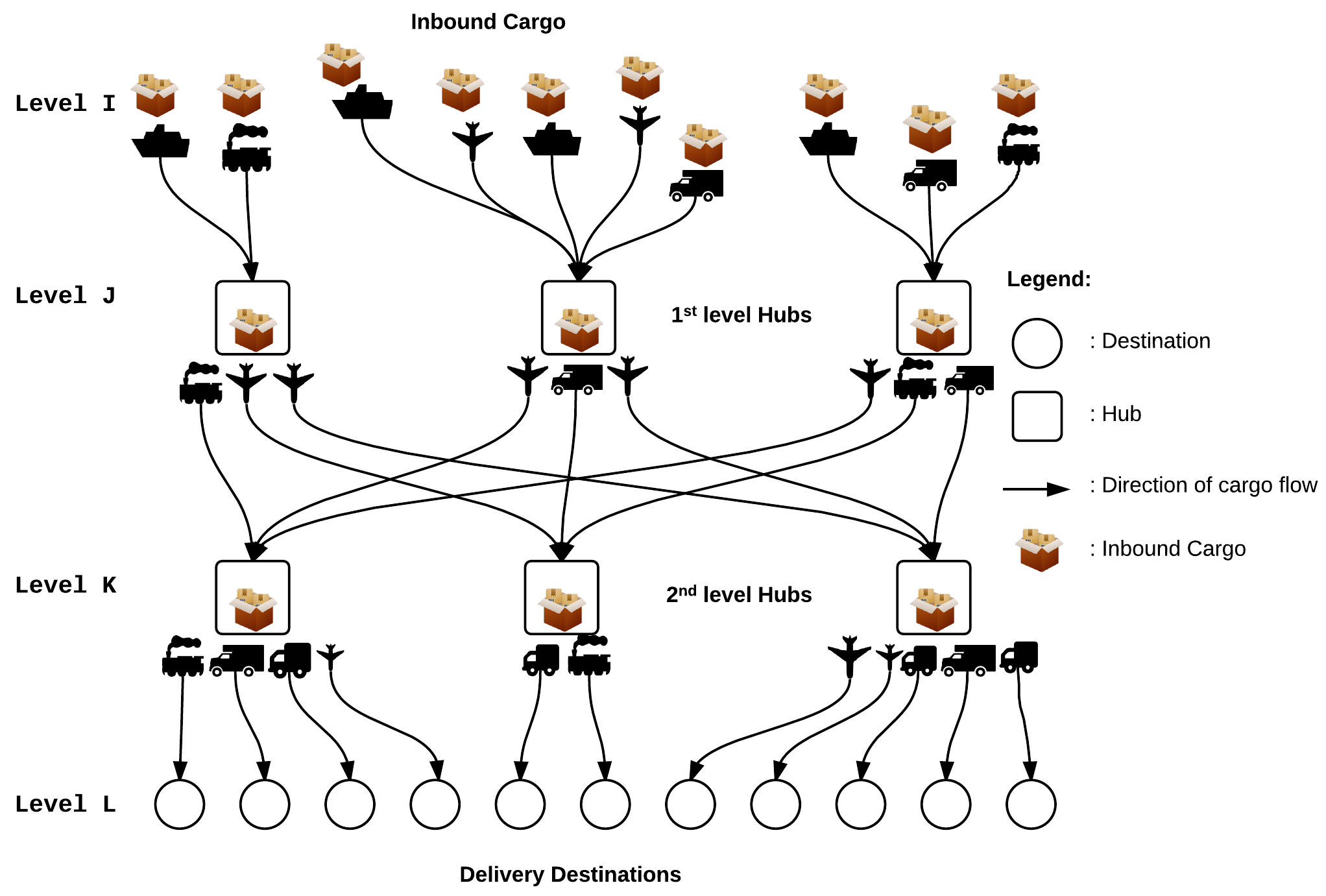}
\caption{Example of a hub-and-spoke freight transportation system with two levels of hubs} 
\label{fig:network}
\end{figure}

It is assumed that at each of the 1$^{st}$ level hubs on level $J$, there are a number of inbound vehicles which have not yet arrived, and there is a vehicle on each outbound route with some cargo already on-board waiting to be dispatched to level $K$. Each of the hubs on level $K$ has a number of delivery routes, and each of the routes has a preplanned schedule for dispatch. Also, there is a scheduled delivery time at the end of each delivery route on level $L$. Another factor taken into account is the travel time and cost between the hubs. This activity consists of the actual transfer between the levels and also the sorting/distribution within the hubs.

Different reliability measures can be defined for a freight transportation system, such as the total amount of cargo that is delivered on-time and the average delay in the system per unit of cargo. All these measures are considered intermediate measures; the ultimate measure is the measure of total cost. The cost function includes operational costs in the system such as transfer cost, waiting cost, storage cost and also penalties for late deliveries. 

We assume that each of the transfers in the system requires a time defined by a probability distribution. In our examples, we use the \achange{truncated} normal distribution, which is one of the most widely used density functions for travel time \citep{aron2014estimating,guo2010multistate}, but our methods are applicable to other probability distributions. The mean and standard deviation of each transfer time can be determined based on recorded data (knowledge base) and real-time parameters such as weather, time of day, traffic congestion and runway queue. Alternatively, a discrete probability model can be used: if we have a record of 20 instances of a given transfer, for example, we can assign a probability of $1/20$ to each of them. 

Each piece of cargo has a scheduled deadline for on-time delivery at its destination, and the dispatch times specified along its route determine the amount of time available for each transfer. Given a set of dispatch times, we can thus compute the probability that the piece of cargo is delivered on time, as well as the expected value of the cost of its delivery (including late penalties).

Clearly, \achange{for every vehicle in the system,} there is a trade-off between waiting for late inbound cargo and immediately dispatching cargo already on-board: as dispatch of a vehicle is delayed, the probability of taking more cargo from the late inbound vehicles increases while the probability of performing the downstream transfers on-time may decrease because \achange{the available time to perform} those transfers may be reduced by waiting.

\section{Our optimization problem} \label{sect-problem}

The main assumptions in our model are these:
\begin{enumerate}
    \item The activities in the system are independent. This independence means that when an activity (e.g. a travel between two hubs) starts, the probability of its completion within a timeframe (e.g. arriving at the next hub) only depends on its own probabilistic parameters. (Even so, activities are entirely interdependent. For example, when dispatch is delayed at a point in system, the probability of success for prior and subsequent activities will increase and decrease, respectively, which is verifiable from the gradient of the objective function.)
    \item \cchange{There are no capacity constraints} 
    for vehicles or for storage of cargo left behind at the hubs.\footnote{\achange{ This is a simplifying assumption, but not unrealistic for large systems. In most of the freight industry, vehicles do not operate near \cchange{capacity;} for examples, ships, trains, planes, trailers \cchange{typically} have load factors around 60 \cchange{or} 70\%\citep{janic2007modelling,mckinnon2009measurement,zhang2018airline}. Hence, the cargo left behind can usually be transported on the next headway with the same unit cost of transport. Any additional cost associated with missing a connection can be incorporated in the storage costs. Capacity of storage facilities is also usually ample in cargo hubs. Considering capacity constraints \cchange{would be a useful extension of our work.} }}
    \item The type of cargo is generic. However, a multi-commodity system can easily be defined and solved as the summation of the same formulation for each commodity. Each part of the summation can then be performed individually and in parallel with others. 
    \item The system can be multimodal or unimodal. 
    \item There are at most two levels of hubs and there are no loops in the network.
    \item Transfer costs among hubs are stationary over time.
    \item The probabilistic characteristics of transfer times (e.g. mean and standard deviation) include the time for transfer and sorting activities within the hubs. 
    \item Late deliveries incur a quantifiable penalty (a fine, loss of future business, etc.). Generally, the penalty increases with the magnitude of lateness, linearly or nonlinearly, but it may become constant after some time, equal to the penalty for delivery on the next headway.
\end{enumerate}

\subsection{Notation}
 Vectors and scalars are in lower case and matrices are in upper case. Bold characters are used for vectors and matrices, and the relevant level on the network is shown as a superscript in parenthesis. A parameter can form a matrix on one level of the network while it forms a vector on another level. We make use of the Hadamard product, a parallelizable operation that makes the computations efficient.

\noindent
$* :$ matrix product \\
$\odot :$ Hadamard product \\
$\bigvee :$ component-wise Max operator \\
$vec() :$ vectorization operator that yields a row vector \\
$\nabla(f,x) :$ Gradient of function $f$ with respect to vector $x$\\
$i \in I:$ index for inbound routes , $\quad n_I = |I|:$ number of inbound routes \\
$j \in J:$ index for upstream hubs , $\quad n_J = |J|:$ number of upstream hubs \\
$k \in K:$ index for downstream hubs , $\quad n_K = |K|:$ number of downstream hubs \\
$l \in L:$ index for delivery routes , $\quad n_L = |L|:$ number of delivery routes \\
$\kappa :$ total cost $(\$)$ \\
$\mu, \sigma :$ mean and standard deviation of a distribution $(hrs)$ \\
$\pi^l :$ penalty function for late delivery $(\$/lb)$ \\
$\pi^m :$ penalty for missing a connection $(\$/lb)$ \\
$\pi^s :$ unit cost of storage for left behind cargo $(\$/lb)$ \\
$\pi^w :$ unit cost of dispatching delay $(\$/lb)$ \\
$c^f :$ transfer cost between two hubs or hubs and destinations $(\$)$ \\
$c^l :$ penalty cost for late delivery $(\$)$ \\
$f(t|\mu,\sigma) :$ probability density function with mean $\mu$ and standard deviation $\sigma$ \\
$\boldsymbol{\O}_{m,n}:$ $m \times n$ matrix of zeros \\
$\boldsymbol{J}_{m,n} :$ $m \times n$ matrix of ones \\
$m :$ network connectivity between two levels (binary)\footnote{\achange{ It defines which nodes in the system are connected to each other, since a network can have missing links between some of its hubs.} }\\
$P(X \leq Y):$ probability of random variable $X$ being less than $Y$ \\
$p :$ probability of successful transfer of cargo between two levels  $\in (0,1)$ \\
$t^p :$ pre-scheduled time of dispatches $(hr)$ \\
$t^s :$ scheduled delivery time on level $L$ $(hr)$ \\
$w^m :$ weight of each cargo element on level I destined to level L $(lb)$ \\
$w^p :$ weight of cargo present at hubs ready to be dispatched  $(lb)$ \\
$t^d :$ actual time of vehicle dispatch from hubs $(hr)$

\subsection{Formulae}
The \achange{decision} variables in our problem are the dispatch times $\boldsymbol{t}^d$. Time is measured forward from the present time, $t = 0$, so the variables are constrained to be non-negative. 

Our objective is to choose dispatch times to minimize the total cost of the freight deliveries. We partition the cost into four terms, highlighted in color in our objective function.
\begin{enumerate}
	\item The penalty for cargo that is delivered, but later than scheduled. (\textcolor{brown}{brown})
	\item The costs associated with cargo that is not delivered because it missed a transfer, including storage costs and penalty for late delivery. (\textcolor{blue}{blue})
    \item The cost associated with delaying a dispatch past its preplanned schedule. It might, for example, include overtime pay for the staff. (\textcolor{red}{red})
    \item The costs of transfers of cargo within the system. (\textcolor{green}{green})
\end{enumerate}
The \textcolor{brown}{first} and \textcolor{blue}{second} terms are costs associated with individual pieces of cargo, while the \textcolor{red}{third} and \textcolor{green}{fourth} terms are costs associated with dispatch of vehicles. We first consider the \textcolor{brown}{first} and \textcolor{blue}{second} terms.

All the costs that cargo  can incur are weight dependent; therefore, we characterize cargo by its weight. As shown in Figure \ref{fig:network}, there are three categories of cargo in the system: 
             inbound cargo $\boldsymbol{W}_{n_I,n_L}^{m,(I \rightarrow L)}$ from level $I$,    
             cargo $\boldsymbol{W}_{n_J,n_L}^{p,(J \rightarrow L)}$ initially present at the hubs on level $J$, 
             and cargo $\boldsymbol{w}_{1,n_L}^{p,(K \rightarrow L)}$ initially present at the hubs on level $K$.
Three scenarios are possible:
\begin{enumerate}
    \item Cargo successfully makes all the connecting hub transfers and arrives on time at the destination. In this ideal scenario, the only cost is the transfer cost.
    \item Cargo successfully makes all the connecting hub transfers but arrives at level $L$ later than scheduled. Added to the transfer cost is a late delivery penalty, determined by a nonlinear penalty function $\pi^I$, monotonically non-decreasing with respect to delay.
    \item Cargo misses a connecting hub transfer on level $J$ or $K$ and is left behind. In this worst case, cargo is stored at the hub until the next dispatch on the same route. The storage cost and late delivery penalty is added to the transfer cost. The penalty $\pi^l$ varies based on the destination on level $L$ but is bounded above by $\pi^m$, the penalty for delivery on the next headway.
\end{enumerate}
In order to calculate the total cost, we first define the probabilities of successful transfers. The probability that cargo on level $I$ arrives in time on level $J$ to successfully make the hub transfer towards level $K$ is
\begin{equation} \label{eq-p1}
\begin{split}
\boldsymbol{P}_{n_I,n_K}^{(I \rightarrow J)} \big( \boldsymbol{T}_{n_J,n_K}^{d,(J \rightarrow K)} \big) 
 &= \mathlarger{P} \big( \boldsymbol{X} \leq \boldsymbol{M}_{n_I,n_J}^{(I \rightarrow J)} * \boldsymbol{T}_{n_J,n_K}^{d,(J \rightarrow K)} \big)  \\
 &= \mathlarger{\int}_{\boldsymbol{\O}_{n_I,n_K}}^{\bs{M}_{n_I,n_J}^{(I \rightarrow J)} * \boldsymbol{T}_{n_J,n_K}^{d,(J \rightarrow K)}}  f \big( t|\{\boldsymbol{\mu}_{1,n_I}^{(I \rightarrow J)}, \boldsymbol{\sigma}_{1,n_I}^{(I \rightarrow J)} \}^T * \boldsymbol{J}_{1,n_K} \big) {\mathrm d} t \, .
\end{split}
\end{equation}
\achange{In this equation, $\mathlarger{P}$ is the probability function that computes whether each of the transfers from level $I$ to $J$ will be completed before the vehicles on level $J$ are dispatched at $\boldsymbol{T}_{n_J,n_K}^{d,(J \rightarrow K)}$. $\bs{X}$ is used as a random variable symbol representing the completion time of transfers from level $I$ to $J$. To elaborate further on this equation, consider any piece of cargo dispatched from level $I$. We want to compute the probability that the piece of cargo makes its way to the vehicle on level $J$ \cchange{before that vehicle is dispatched.}
The transfer from level $I$ to $J$ including the inter-hub operations has mean $\bs{\mu}_{1,n_I}^{(I \rightarrow J)}$ and standard deviation $\bs{\sigma}_{1,n_I}^{(I \rightarrow J)} $. The probability of cargo reaching the vehicle on level $J$ anytime before its dispatch (limits of the integral) is the probability of successful transfer.}

Similarly,  the probability that cargo on level $J$ arrives in time on level $K$ to successfully make the hub transfer towards level $L$ is
\begin{equation} \label{eq-p2}
\begin{split}
\boldsymbol{P}&_{n_J,n_L}^{(J \rightarrow K)} \big( \boldsymbol{T}_{n_J,n_K}^{d,(J \rightarrow K)}, \boldsymbol{t}_{1,n_L}^{d,(K \rightarrow L)} \big) \\
&= \mathlarger{P} \big( \boldsymbol{X} \leq \boldsymbol{J}_{n_J,1} * \boldsymbol{t}_{1,n_L}^{d,(K \rightarrow L)} - \boldsymbol{T}_{n_J,n_K}^{d,(J \rightarrow K)} * \boldsymbol{M}_{n_K,n_L}^{(K \rightarrow L)} \big) \\
&= \mathlarger{\int}_{\boldsymbol{\O}_{n_J,n_L}}^{ \boldsymbol{J}_{n_J,1} * \boldsymbol{t}_{1,n_L}^{d,(K \rightarrow L)} - \boldsymbol{T}_{n_J,n_K}^{d,(J \rightarrow K)} * \boldsymbol{M}_{n_K,n_L}^{(K \rightarrow L)} }  f \big( t|\{\boldsymbol{\mu}_{n_J,n_K}^{(I \rightarrow J)}, \boldsymbol{\sigma}_{n_J,n_K}^{(I \rightarrow J)} \} * \boldsymbol{M}_{n_K,n_L}^{(K \rightarrow L)} \big) {\mathrm d} t \, .
\end{split}
\end{equation}
 To see how expected cost is calculated from these probabilities, consider one element of cargo $w_{i,l}^{m,(I \rightarrow L)}$ that comes from node $i$ on level $I$,  destined to go through hubs $j$ and $k$ on levels $J$ and $K$ to ultimately get delivered at node $l$ on level $L$. 
 With probability $p_{i,k}^{(I \rightarrow J)}p_{j,l}^{(J \rightarrow K)}$ it makes all connections, moving from level $I$ to level $K$.  
 It misses its hub transfer on level $J$ with probability $1 - p_{i,k}^{(I \rightarrow J)}$ and incurs storage cost and late delivery penalty. 
 It  gets stuck on level $K$ and incurs the late penalty and storage cost with probability $p_{i,k}^{(I \rightarrow J)}(1 - p_{j,l}^{(J \rightarrow K)})$. These probabilities sum to 1, so by multiplying them by the corresponding costs we obtain the expected value of cargo costs. We include the expected value of the penalty for late delivery at level $L$ based on the probability of late arrival there:
\begin{equation} \label{eq-p3}
\begin{split}
 \bs{c}_{1,n_L}^{l,(K \rightarrow L)} \big( &\bs{t}_{1,n_L}^{d,(K \rightarrow L)} \big) \\
 &= \mathlarger{\int}_{\bs{t}_{1,n_L}^s - \bs{t}_{1,n_L}^{d,(K \rightarrow L)} }^{+\bs{\infty}_{1,n_L}} \pi^l \big( t - \bs{t}_{1,n_L}^s + \bs{t}_{1,n_L}^{d,(K \rightarrow L)} \big) \: . \: f \big( t|\{\bs{\mu}_{1,n_L}^{(I \rightarrow J)}, \bs{\sigma}_{1,n_L}^{(I \rightarrow J)} \} \big) {\mathrm d} t \, .
\end{split}
\end{equation}
Clearly, $\bs{c}_{1,n_L}^{l,(K \rightarrow L)}$ is strictly positive in scenario 2 and is zero in scenario 1. The total cost associated with $w_{i,l}^{m,(I \rightarrow L)}$ is thus $w_{i,l}^{m,(I \rightarrow L)}$  times
\begin{align*}
 \textcolor{blue}{(1 - p_{i,k}^{(I \rightarrow J)}) (\pi_{1,l}^{m,(J \rightarrow K)}} & \textcolor{blue}{+ \pi_{1,j}^{s,(J \rightarrow K)}) + p_{i,k}^{(I \rightarrow J)}(1 - p_{j,l}^{(J \rightarrow K)} ) (\pi_{1,l}^{m,(K \rightarrow L)} + \pi_{1,k}^{s,(K \rightarrow L)}) } \\
& \textcolor{brown}{+ p_{i,k}^{(I \rightarrow J)} p_{j,l}^{(J \rightarrow K)} c_{1,l}^{l,(K \rightarrow L)}}.
\end{align*}

In order to sum the costs for all the cargo elements we use Hadamard product notation to denote element-by-element multiplication of the product of probability matrices with the connectivity matrices $\bs{M}_{n_I,n_J}^{(I \rightarrow J)}$ and $\bs{M}_{n_K,n_L}^{(K \rightarrow L)}$, which are binary and define how the nodes are connected in the network.

In a similar way, we sum up the costs for delayed dispatch of a vehicle and for transfers of cargo. This gives us a total expected cost to minimize subject to non-negativity constraints on dispatch times:
\begin{equation} \label{eq-obj}
\begin{split} 
&\min\limits_{\bs{T}_{n_J,n_K}^{d,(J \rightarrow K)}, \bs{t}_{1,n_L}^{d,(K \rightarrow L)} } \kappa = \\
& \textcolor{brown}{ \quad \Big[ \; \bs{J}_{1,n_J} * \Big( \big( \boldsymbol{P}_{n_I,n_K}^{(I \rightarrow J)} * \boldsymbol{M}_{n_K,n_L}^{(K \rightarrow L)} \big) \odot \big( \boldsymbol{M}_{n_I,n_J}^{(I \rightarrow J)} * \boldsymbol{P}_{n_J,n_L}^{(J \rightarrow K)} \big) \odot \boldsymbol{W}_{n_I,n_L}^{m,(I \rightarrow L)} \Big) * \boldsymbol{c}_{1,n_L}^{l,(K \rightarrow L)^T} } \\
&\textcolor{brown}{ \quad + \; \boldsymbol{J}_{1,n_J} * \big( \boldsymbol{P}_{n_J,n_L}^{(J \rightarrow K)} \odot \boldsymbol{W}_{n_J,n_L}^{p,(J \rightarrow L)} \big) * \boldsymbol{c}_{1,n_L}^{l,(K \rightarrow L)^T} + \boldsymbol{w}_{1,n_L}^{p,(K \rightarrow L)} * \boldsymbol{c}_{1,n_L}^{l,(K \rightarrow L)^T} \Big]} \\
&+ \textcolor{blue}{ \bigg[ \big( \boldsymbol{\pi}_{1,n_J}^{s,(J \rightarrow K)} * \boldsymbol{M}_{n_I,n_J}^{(I \rightarrow J)^T} \big) * \Big( \boldsymbol{J}_{n_I,n_L} - \big( \boldsymbol{P}_{n_I,n_K}^{(I \rightarrow J)} * \boldsymbol{M}_{n_K,n_L}^{(K \rightarrow L)} \big) \Big) } \\
& \textcolor{blue}{ \quad \odot \; \boldsymbol{W}_{n_I,n_L}^{m,(I \rightarrow L)} \bigg) * \boldsymbol{\pi}_{1,n_L}^{m,(J \rightarrow K)^T} + \bigg( \boldsymbol{J}_{1,n_I} * \Big( \big( \boldsymbol{P}_{n_I,n_K}^{(I \rightarrow J)} * \boldsymbol{M}_{n_K,n_L}^{(K \rightarrow L)} \big) } \\
& \textcolor{blue}{ \quad \odot \; \big( \boldsymbol{J}_{n_I,n_L} - \boldsymbol{M}_{n_I,n_J}^{(I \rightarrow J)} * \boldsymbol{P}_{n_J,n_L}^{(J \rightarrow K)} \big) \odot \boldsymbol{W}_{n_I,n_L}^{m,(I \rightarrow L)} \Big) + \; \boldsymbol{J}_{1,n_J} } \\
&\textcolor{blue}{\quad * \Big( \big( \boldsymbol{J}_{n_J,n_L} - \boldsymbol{P}_{n_J,n_L}^{(J \rightarrow K)} \big) \odot \boldsymbol{W}_{n_J,n_L}^{p,(J \rightarrow L)} \Big) \bigg) * \big( \boldsymbol{\pi}_{1,n_K}^{s,(K \rightarrow L)} * \boldsymbol{M}_{n_K,n_L}^{(K \rightarrow L)} + \boldsymbol{\pi}_{1,n_L}^{m,(K \rightarrow L)} \big)^T \bigg] } \\
&+ \textcolor{red}{ \Big[ \; \bs{J}_{1,n_J} * \Big( \bigvee \big( \bs{T}_{n_J,n_K}^{d,(J \rightarrow K)} - \bs{T}_{n_J,n_K}^{p,(J \rightarrow K)},\textbf{\O}_{n_J,n_K} \big) \odot \bs{\Pi}_{n_J,n_K}^{w,(J \rightarrow K)} \Big) * \bs{J}_{n_K,1} }\\
&\textcolor{red}{ \quad + \; \bs{\pi}_{1,n_L}^{w,(K \rightarrow L)} * \bigvee \big( \bs{t}_{1,n_L}^{d,(K \rightarrow L)} - \bs{t}_{1,n_L}^{p,(K \rightarrow L)}, \textbf{\O}_{1,n_L} \big)^T \Big] } \\
&+ \textcolor{green}{ \Big[ \; \bs{J}_{1,n_J} * \bs{C}_{n_J,n_K}^{f,(J \rightarrow K)} * \bs{J}_{n_K,1} + \bs{c}_{1,n_L}^{f,(K \rightarrow L)} * \bs{J}_{n_L,1} \Big]} ,
\end{split}
\end{equation}
subject to: 
\begin{equation} \label{eq-constraint}
    \bs{T}_{n_J,n_K}^{d,(J \rightarrow K)} \: , \: \bs{t}_{1,n_L}^{d,(K \rightarrow L)} \geq 0 \, .
\end{equation}
Brackets in equation (\ref{eq-obj}) distinguish each of the four terms in the objective function. The third \textcolor{red}{red} term  is the cost associated with delaying the dispatches which increases with delay in dispatch. $\boldsymbol{T}_{n_J,n_K}^{d,(J \rightarrow K)}$ is the time that each vehicle on level $J$ is dispatched towards level $K$ and $\boldsymbol{T}_{n_J,n_K}^{p,(J \rightarrow K)}$ is the pre-planned dispatch schedule for each of those vehicles. $\boldsymbol{\Pi}_{n_J,n_K}^{w,(J \rightarrow K)}$ is the cost ($\$$ per hour of delay) for delaying the dispatch for each of the vehicles on level $J$. The max operator is used to ensure that dispatching ahead of schedule is not rewarded. There is only a single point of non-differentiability for each element of $\boldsymbol{t}^d$, and we can rely on the subgradient for these points.

The last \textcolor{green}{green} term accounts for costs of transferring the cargo within the system to the destinations. The $\boldsymbol{c}^f$ can be a constant cost for each of the transfers in the system, or it can be defined as a function of time and other parameters. It would be practical to consider $\boldsymbol{c}^f$ as a function that depends on how fast a transfer between two nodes is performed. In other words, the transfer cost can be interdependent with characteristics of probability distribution for the corresponding transfer time. In such a setting, transfers in the network can be accelerated with increased $\boldsymbol{c}^f$.

\section{Our homotopy optimization algorithm} \label{sect-alg}

\achange{ The optimization problem defined in previous section can be solved with off-the-shelf nonlinear non-convex optimization solvers that can handle linear constraints. 
Since, our problem has to be solved in real-time, both the quality of the solution and the time it takes to find it are essential.
Here, we design a homotopy algorithm to solve our problem and later compare the \cchange{quality of solutions obtained by it \cchange{to solutions from} off-the-shelf solvers, equalizing computation cost}. We note that designing \cchange{customized} nonlinear optimization algorithms is not common in \cchange{the} freight transportation literature.

To design the homotopic optimization algorithm, we first \cchange{identify some useful properties of the objective function.} }

For convenience, let  $\boldsymbol{t}^V$ denote the vector of variables, with 
$vec( \boldsymbol{T}_{n_J,n_K}^{d,(J \rightarrow K)})$ 
ordered first, followed by $\boldsymbol{t}_{1,n_L}^{d,(K \rightarrow L)}$.
The examples in section \ref{sect-results} demonstrate that the objective function for our problem is non-convex.
Since the objective function is twice differentiable almost everywhere, and since the constraints are just non-negativity, a gradient projection algorithm is a good candidate for a base algorithm, but we enhance it in order to try to find a local minimum close to the global one.
We express our objective function as
\begin{equation} \label{eq-alg1}
    \kappa (\boldsymbol{t}^V) = g(\boldsymbol{t}^V) + h(\boldsymbol{t}^V),
\end{equation}
denoting the sum of the \textcolor{brown}{first} and \textcolor{red}{third} terms in (\ref{eq-obj})
 by $g(\boldsymbol{t}^V)$
and the sum of the \textcolor{blue}{second} and \textcolor{green}{fourth} terms by $h(\boldsymbol{t}^V)$. This partitioning is motivated by the following observations.

\begin{lemma} \label{lemma1}
$g(\boldsymbol{t}^V)$ is monotonically non-decreasing and unbounded with respect to $\boldsymbol{t}_{1,n_L}^{d,(K \rightarrow L)}$. $h(\boldsymbol{t}^V)$ is non-negative and monotonically non-increasing with respect to $\boldsymbol{t}_{1,n_L}^{d,(K \rightarrow L)}$. If the derivatives of the probability density functions converge to zero as $t \rightarrow \infty$, or if the density functions have finite support, then the gradient $ \nabla(h,\boldsymbol{t}_{1,n_L}^{d,(K \rightarrow L)})$ converges to zero.
\end{lemma}
\begin{proof}
The  \textcolor{red}{red} term in $g(\boldsymbol{t}^V)$ increases linearly with $\boldsymbol{t}_{1,n_L}^{d,(K \rightarrow L)}$. For the  \textcolor{brown}{brown} term,  $\boldsymbol{t}_{1,n_L}^{d,(K \rightarrow L)}$ is the upper limit for integration for the nonnegative functions $\boldsymbol{P}_{n_J,n_L}^{(J \rightarrow K)}$ and $\boldsymbol{c}_{1,n_L}^{l,(K \rightarrow L)}$,  so increasing $\boldsymbol{t}_{1,n_L}^{d,(K \rightarrow L)}$ cannot cause the $g(\boldsymbol{t}^V)$ to decrease. This proves the first statement of the lemma.

Note that $h(\boldsymbol{t}^V)$ is a function of $\boldsymbol{t}_{1,n_L}^{d,(K \rightarrow L)}$ through $\boldsymbol{P}_{n_J,n_L}^{(J \rightarrow K)}$. In contrast to $g(\boldsymbol{t}^V)$, all instances of $\boldsymbol{P}_{n_J,n_L}^{(J \rightarrow K)}$ have negative signs here. With same reasoning as for $g(\boldsymbol{t}^V)$, increasing the $\boldsymbol{t}_{1,n_L}^{d,(K \rightarrow L)}$ cannot decrease the $\boldsymbol{P}_{n_J,n_L}^{(J \rightarrow K)}$ and therefore cannot increase $h(\boldsymbol{t}^V)$.

We can make $h(\boldsymbol{t}^V)$ arbitrary close to constant by increasing $\boldsymbol{t}_{1,n_L}^{d,(K \rightarrow L)}$ enough, since  $\boldsymbol{P}_{n_I,n_K}^{(I \rightarrow J)}$ and $\boldsymbol{P}_{n_J,n_L}^{(J \rightarrow K)}$ can be made arbitrarily close to their upper bound,~$1$. 
This means $ \nabla(h,\boldsymbol{t}_{1,n_L}^{d,(K \rightarrow L)} ) \rightarrow 0$. 

Setting $\boldsymbol{P}_{n_I,n_K}^{(I \rightarrow J)} \approx 1$ and $\boldsymbol{P}_{n_J,n_L}^{(J \rightarrow K)} \approx 1$ will make the minimum of $h(\boldsymbol{t}^V)$ a positive function of nonnegative costs and therefore $h(\boldsymbol{t}^V)$ itself is non-negative.
\end{proof}

\achange{
We leverage the result of this lemma to transform our objective function into a state where the global minimizer is known. }Homotopy (continuation) methods, rather than directly dealing with the given optimization problem, create a sequence of related problems, ending with the given problem, for which the global solution to the first problem is known or can be found easily. Then starting from this problem and its known solution, we step our way to the given problem, tracing a path of solutions \citep{dunlavy2005homotopy,nocedal2006numerical}. This optimization method has proven to be effective in many very difficult optimization problems, e.g., \cite{chapelle2006continuation,floudas2014recent,guddat2006modified,nocedal2006numerical,tuy2000monotonic,watson1989modern,wu1996effective,dunlavy2005homotopy,dunlavy2005hope,mobahi2015link,mobahi2015theoretical,anandkumar2017homotopy}.

\achange{
In our homotopy method, we first compute the set of variables ${\bs{t}^V}^0$, that would minimize the $h(\bs{t}^V)$. Then, we transform the $g(\bs{t}^V)$ such that it is minimized at ${\bs{t}^V}^0$, too. Adding the $h(\bs{t}^V)$ to the transformed $g(\bs{t}^V)$ would yield the transformed objective function with global minimizer at ${\bs{t}^V}^0$. Following is the formal procedure that achieves this.
}

Given a parameter $\gamma > 0$, define
\begin{equation} \label{eq-alg6}
\boldsymbol{\delta}_{n_J,n_K}^{(J \rightarrow K)} = \bigg( \bigvee \Big[ \boldsymbol{M}_{n_I,n_J}^{(I \rightarrow J)} \odot \Big( \big( \boldsymbol{\mu}_{1,n_I}^{(I \rightarrow J)} + \gamma \; \boldsymbol{\sigma}_{1,n_I}^{(I \rightarrow J)} \big)^T * \boldsymbol{J}_{1,n_J} \Big) \Big]  \bigg)^T * \boldsymbol{J}_{1,n_K},
\end{equation}
\begin{equation} \label{eq-alg7}
\begin{split}
\boldsymbol{\delta}_{1,n_L}^{(K \rightarrow L)} = \bigvee \Big[ \big( \boldsymbol{\mu}_{1,n_I}^{(I \rightarrow J)} + \gamma \; &\boldsymbol{\sigma}_{1,n_I}^{(I \rightarrow J)} \big)^T * \boldsymbol{J}_{1,n_L} \\
&+ \boldsymbol{M}_{n_I,n_J}^{(I \rightarrow J)} * \big( \boldsymbol{\mu}_{n_J,n_K}^{(J \rightarrow K)} + \gamma \; \boldsymbol{\sigma}_{n_J,n_K}^{(J \rightarrow K)} \big) * \boldsymbol{M}_{n_K,n_L}^{(K \rightarrow L)} \Big] \, .
\end{split}
\end{equation}
Setting $\gamma$ large enough (e.g., $\gamma \approx 5$ for the normal distribution) ensures that at time
\begin{equation} \label{eq-alg8}
    {\boldsymbol{t}^V}^0 = \Big\langle vec \Big( \bigvee \big( \bs{\delta}_{n_J,n_K}^{(J \rightarrow K)} , \bs{T}_{n_J,n_K}^{p,(J \rightarrow K)} \big) \Big), \bigvee \big( \bs{\delta}_{1,n_L}^{(K \rightarrow L)} , \bs{t}_{1,n_L}^{p,(K \rightarrow L)} \big) \Big\rangle,
\end{equation}
$\bs{P}_{n_I,n_K}^{(I \rightarrow J)} \approx 1$ and $\bs{P}_{n_J,n_L}^{(J \rightarrow K)} \approx 1$. So at that time, the function $h({\bs{t}^V}^0)$ will be close to its minimum value. We use this as a starting point for our homotopy algorithm.

\achange{
We now transform the $g(\bs{t}^V)$, such that its value at ${\bs{t}^V}^0$ is close to its minimum value. To achieve this, we define }
\begin{equation} \label{eq-alg9}
\bs{\delta} = - \Big\langle \textbf{\O}_{1, n_J \times n_K} \; , \; \bs{\delta}_{1,n_L}^{(K \rightarrow L)} + \bs{\mu}_{1,n_L}^{(K \rightarrow L)} + \gamma \; \bs{\sigma}_{1,n_L}^{(K \rightarrow L)} - \bs{t}_{1,n_L}^s \Big\rangle,
\end{equation}
\achange{and replace $g(\bs{t}^V)$ with $g(\bs{t}^V + \bs{\delta})$.

To verify that this transformation brings the $g({\bs{t}^V}^0 + \bs{\delta})$ close to its minimum value, consider the third term, $f \big( t|\{\bs{\mu}_{1,n_L}^{(I \rightarrow J)}, \bs{\sigma}_{1,n_L}^{(I \rightarrow J)} \} \big)$, in equation \eqref{eq-p3}. Choosing $\gamma$ large enough ensures that the value of this term is close to zero within the transformed integral limits of equation \eqref{eq-p3}. This leads to $\bs{c}_{1,n_L}^{l,(K \rightarrow L)} \approx 0$, and hence, the first term of $g({\bs{t}^V}^0 + \bs{\delta})$ will be close to its minimal value, zero. The second term of $g({\bs{t}^V}^0 + \bs{\delta})$ is minimized to zero, by construction of ${\bs{t}^V}^0$ in equation \eqref{eq-alg8}.

The transformed objective function is the sum of two components, $h(\bs{t}^V)$ and $g(\bs{t}^V + \bs{\delta})$. Since ${\bs{t}^V}^0$ is the global minimizer of both components, it is also the global minimizer of the transformed objective function.
}

Starting from this transformed system and its known global minimizer, we iteratively transform the system back to its original form, using the minimizer from the previous iteration as the starting point for the next iteration.
Algorithm \ref{alg-homo} is the homotopy algorithm for our problem. This algorithm is not guaranteed to converge to the global minimizer of the objective function. However, as shown in section \ref{sect-results}, it can find considerably better minimizers compared to some of the well-known global optimization solvers because it is designed specifically for this problem. 
\IncMargin{1em}
\begin{algorithm}[h!] \caption{\small{Homotopy algorithm for optimizing dispatch times}} \label{alg-homo}
\SetAlgoLined
\SetKwData{Left}{left}\SetKwData{This}{this}\SetKwData{Up}{up}
\SetKwInOut{Input}{Input}\SetKwInOut{Output}{Output}
\Input{ Real-time instance of the freight system with all its parameters }
\Output{ Optimal values for time variables $\boldsymbol{t}^V$ and the minimized cost}
 Find the starting point ${\boldsymbol{t}^V}^0$ using equations (\ref{eq-alg6})-(\ref{eq-alg8})\;
 Calculate $\boldsymbol{\delta}$ using equation (\ref{eq-alg9})\;
 Choose the total number of iterations $N$ sufficiently large, e.g. 20\;
 \For{$i\leftarrow 1$ \KwTo $N$}{
   Calculate homotopic transformation of function $g(\boldsymbol{t}^V)$ using $g^k = g(\boldsymbol{t}^V + \frac{N-k}{N} \boldsymbol{\delta} )$\;
   Perform  gradient projection (or another gradient-based method) on ${\kappa}^k = g^k + h$ from starting point ${\boldsymbol{t}^V}^{(k-1)}$ to find a minimizer satisfying the $2^{nd}$ order optimality conditions and save the minimizer as ${\boldsymbol{t}^V}^k$\; }
  Return ${\boldsymbol{t}^V}^N$ as the optimal solution and $\kappa^{N}({\boldsymbol{t}^V}^N)$ as the minimized cost\;
\end{algorithm}

\section{Extension to bidirectional networks and mixed hierarchies} \label{sect-extention}

Our formulation and algorithm can also be used on networks with bidirectional flow and mixed hierarchies.

\begin{figure}[H]
\centering
\includegraphics[width=0.5\textwidth]{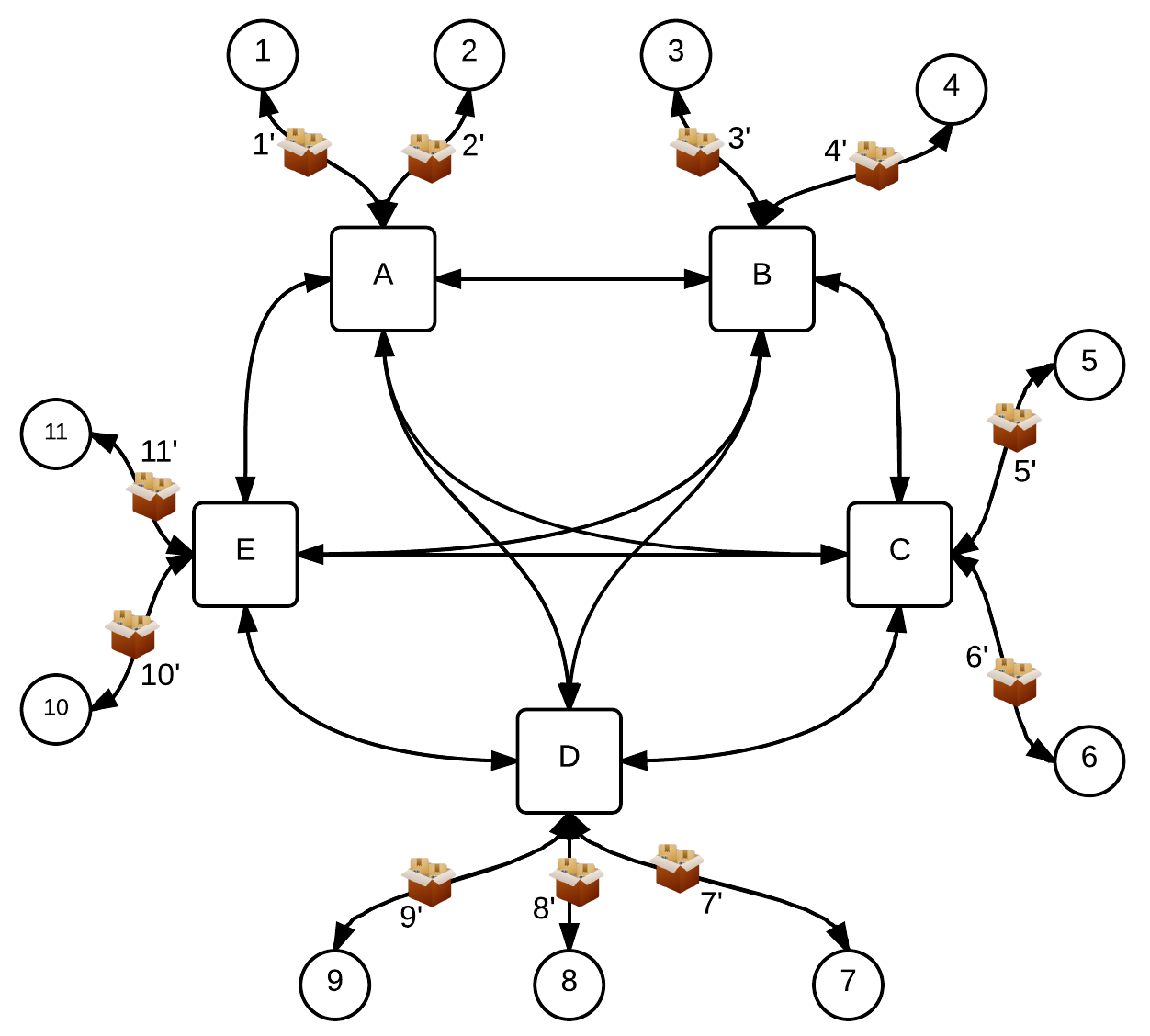}
\caption{A bidirectional freight system}
\label{fig:bi}
\end{figure}
\begin{figure}[H]
\centering
\includegraphics[width=0.55\textwidth]{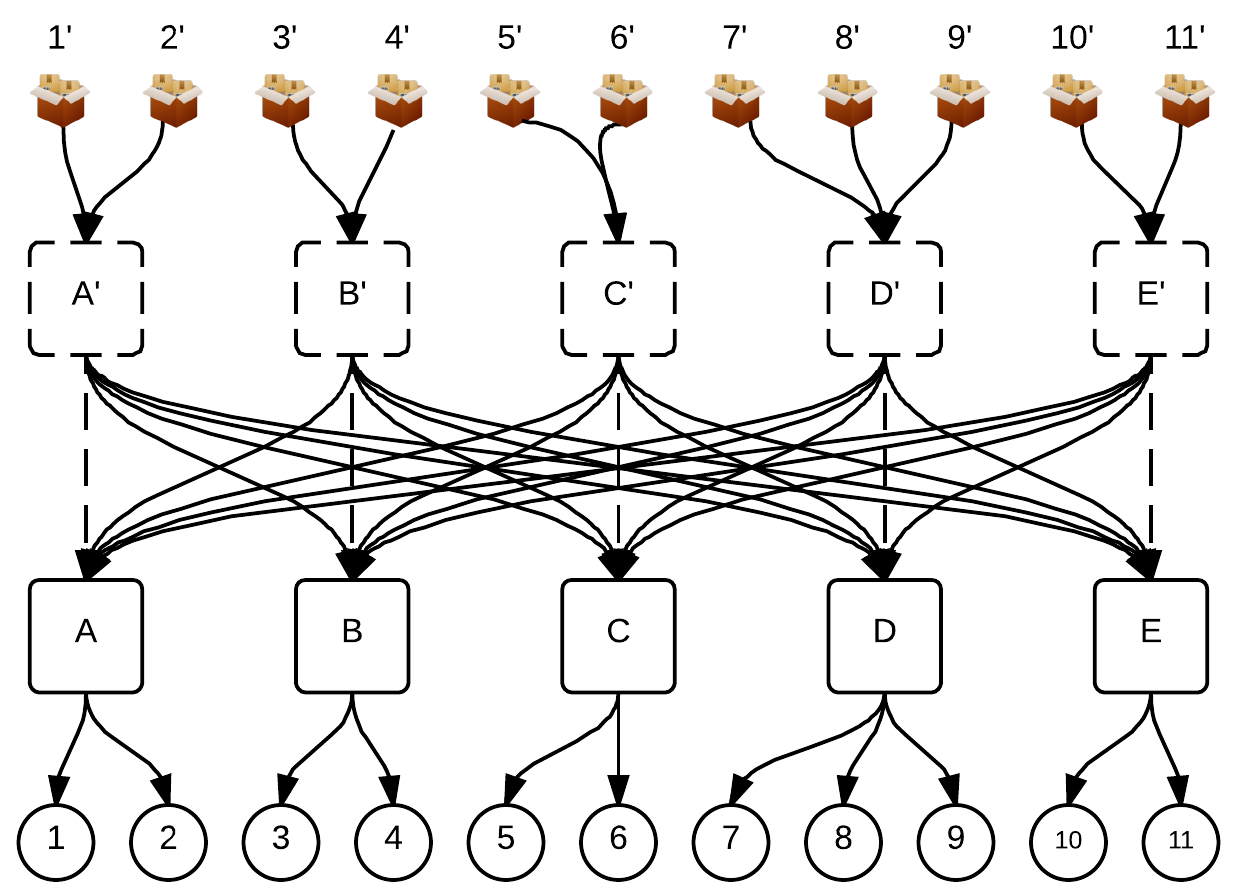}
\caption{Alternative representation of network in Figure \ref{fig:bi}}
\label{fig:bi-alt}
\end{figure}

For bidirectionality, consider the network in Figure \ref{fig:bi} as an example. The cargo moves in both directions on all routes and spokes and the maximum level of hierarchy is assumed to be 2 in any direction. This means that any piece of cargo passes through at most 2 hubs before reaching its destination. This network can be replaced by the unidirectional network in Figure \ref{fig:bi-alt} which  has the same dispatch variables as the original network, so Algorithm \ref{alg-homo} can be applied.

For mixed hierarchies of 2 and less, consider the example in Figure \ref{fig:mixed}. In this network, destinations $e$ and $m$ receive cargo  processed in two, one or zero hubs prior to delivery. This kind of network can be divided into two separate and independent networks, one containing hierarchies of 2 and one containing hierarchies of less than 2, as shown in Figures \ref{fig:mixedalt1} and \ref{fig:mixedalt2}.

\begin{figure}[H]
\centering
\includegraphics[width=0.55\textwidth]{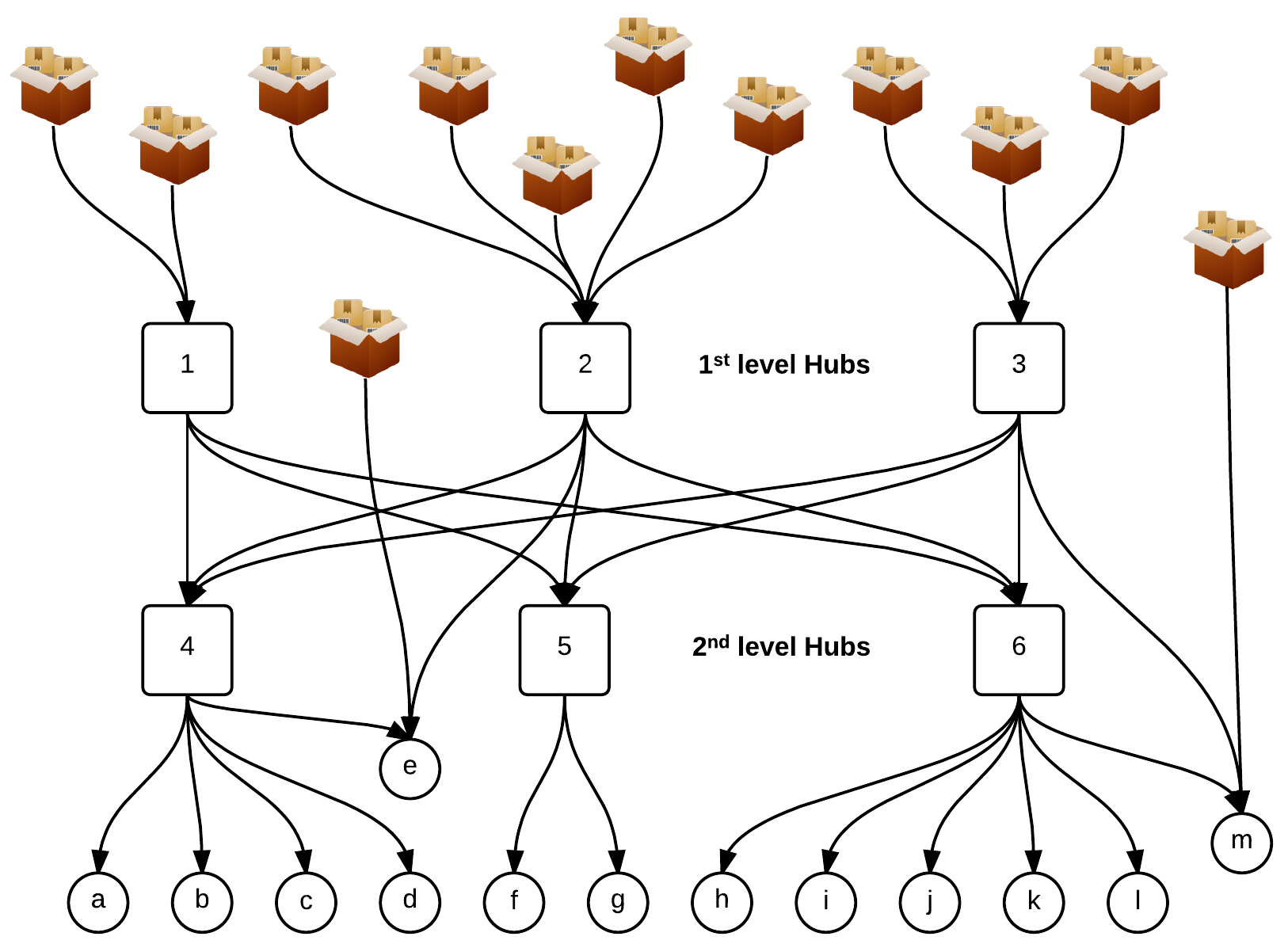}
\caption{A system with mixed hierarchies}
\label{fig:mixed}
\end{figure}

\begin{figure}[H]
\centering
\begin{subfigure}{0.5\textwidth}
\includegraphics[width=1\linewidth]{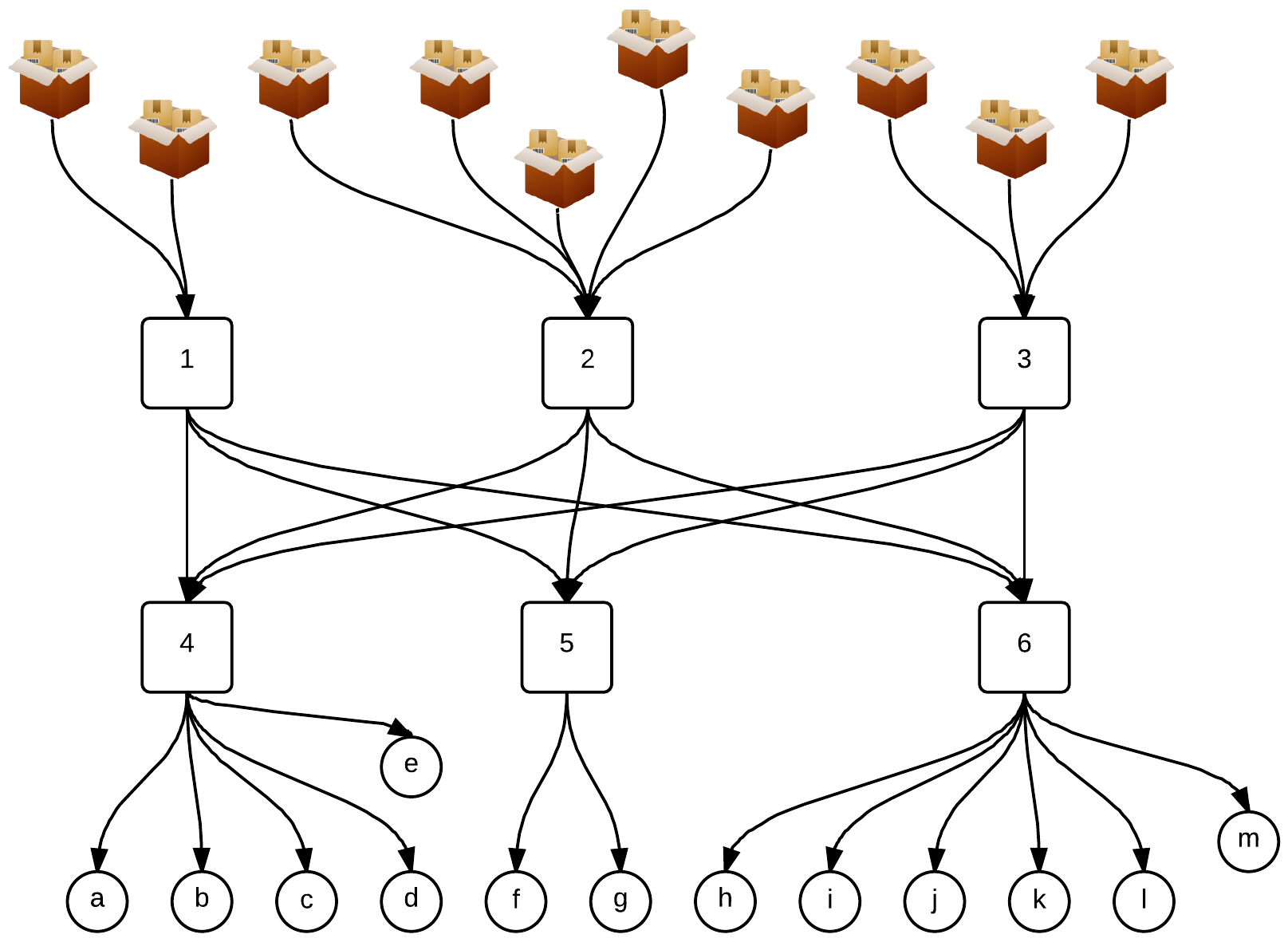}
\caption{hierarchies of 2}
\label{fig:mixedalt1}
\end{subfigure}
\begin{subfigure}{0.4\textwidth}
\includegraphics[width=1\linewidth]{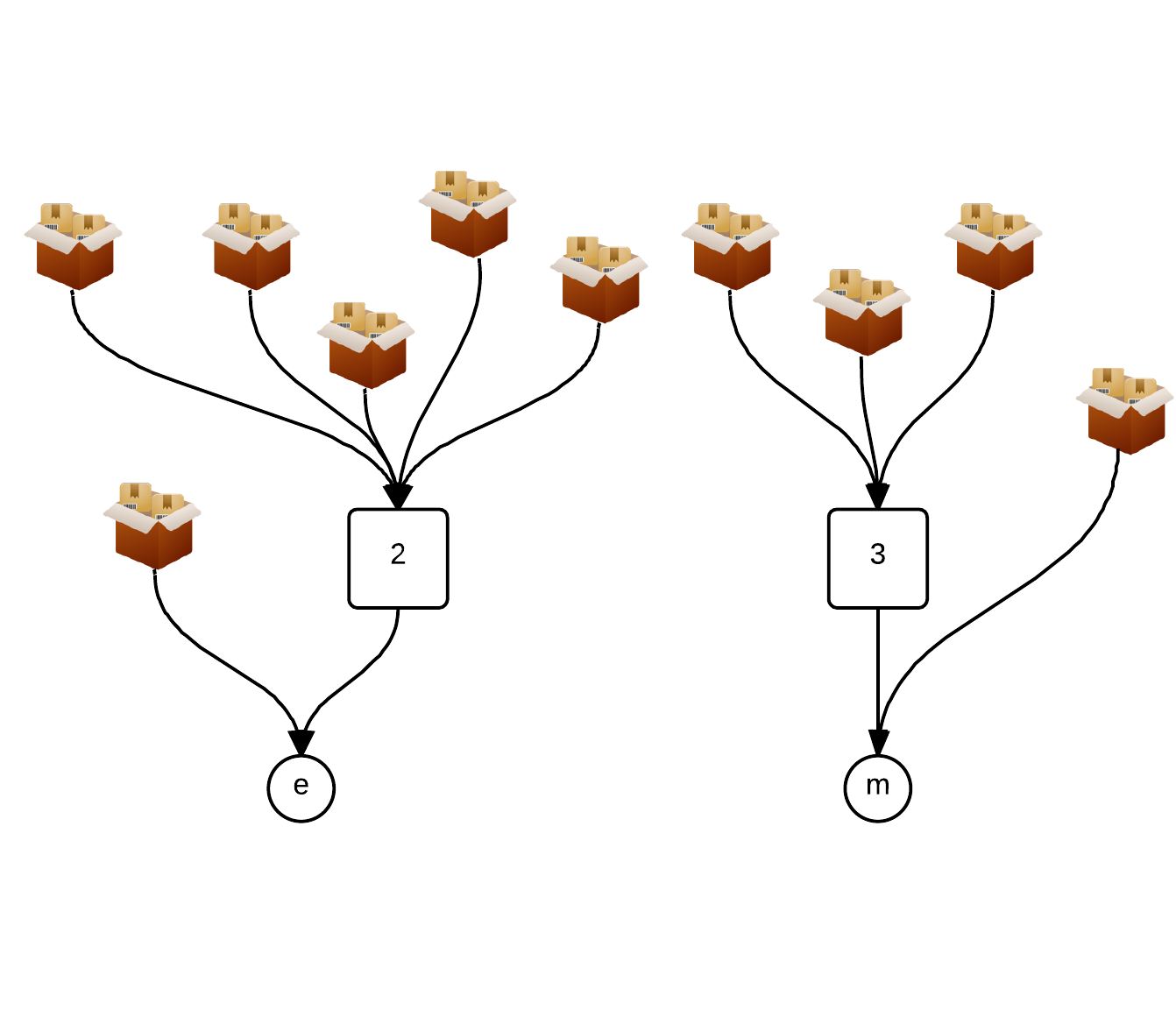}
\caption{hierarchies of less than 2}
\label{fig:mixedalt2}
\end{subfigure}
\caption{Decomposition of system in Figure \ref{fig:mixed}}
\label{fig:mixedalt}
\end{figure}

\section{Results} \label{sect-results}
We demonstrate our model framework and algorithm on three examples.

\subsection{ Case 1: A small system } 
Figure \ref{fig:case1} shows a network with only two variables, dispatch times for vehicles on level $J$ and level $K$, respectively. We investigate this problem to gain insight into the model and to demonstrate how the homotopy algorithm operates. The contours of the objective function are shown in Figure \ref{fig:case1contour}. Algorithm \ref{alg-homo} finds the global minimizer  $\boldsymbol{t}^V=[4.08 ,10.90]$, marked with a \textcolor{red}{red x},  corresponding to $\kappa=\$37,100$. There are several saddle points and local minimizers. The value of objective function for the second-best minimizer is $\kappa=\$41,960$.

Figure \ref{fig:case1homotopy} shows the  shape of the objective function in each of the $N=6$ iterations of the algorithm. Each \textcolor{red}{red x} on the sub-graphs is the minimizer ${\boldsymbol{t}^V}^k$ found by Algorithm \ref{alg-homo}. The final solution is the global minimizer of the original problem.

When provided with a good starting point and adequate time, most solvers can deliver the global optimizer for small problems like this one.

\begin{figure}[H]
\centering
\includegraphics[width=0.6\textwidth]{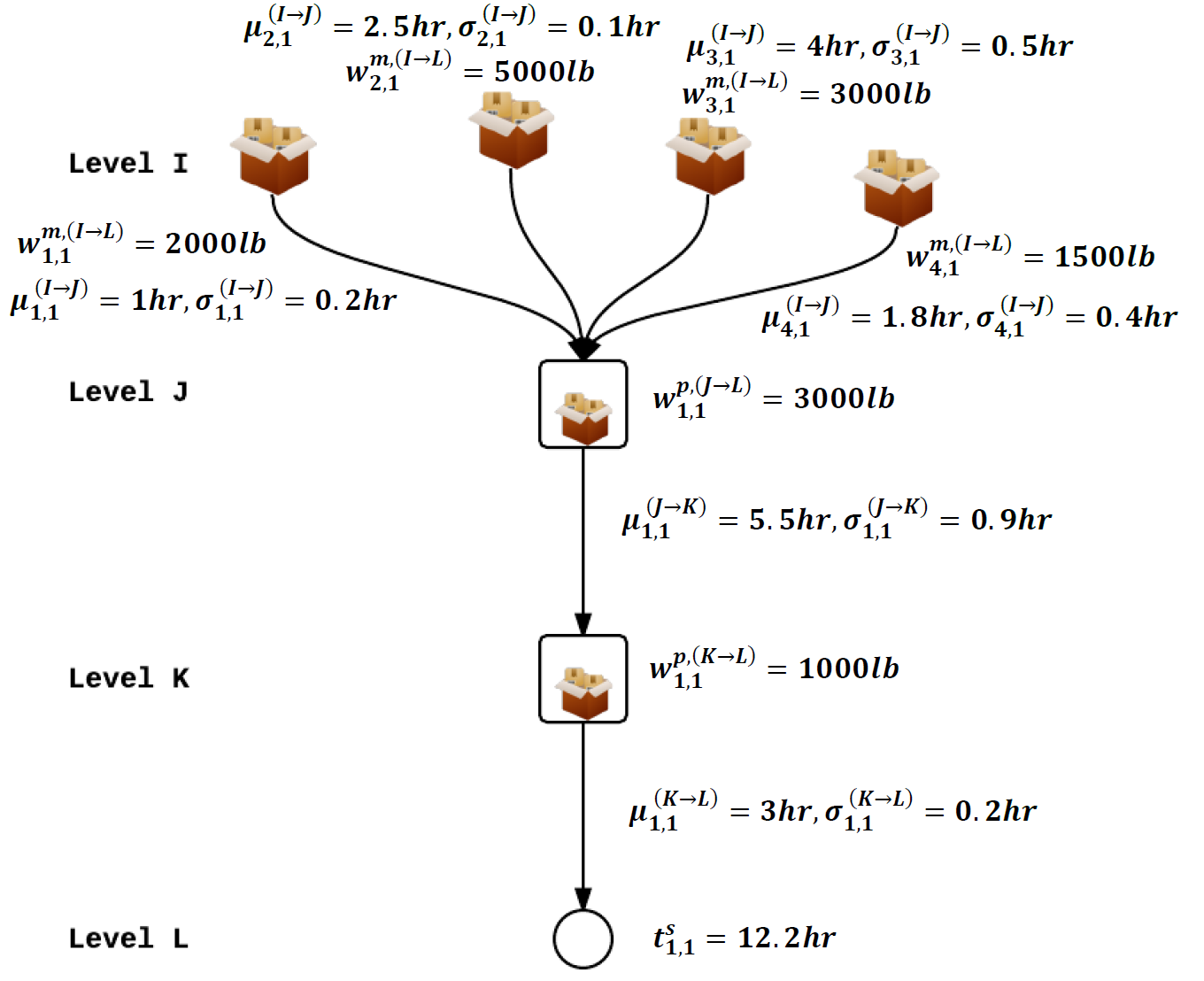}
\caption{Network in Case 1} 
\label{fig:case1}
\end{figure}

\begin{figure}[H]
\centering
\includegraphics[width=0.6\textwidth]{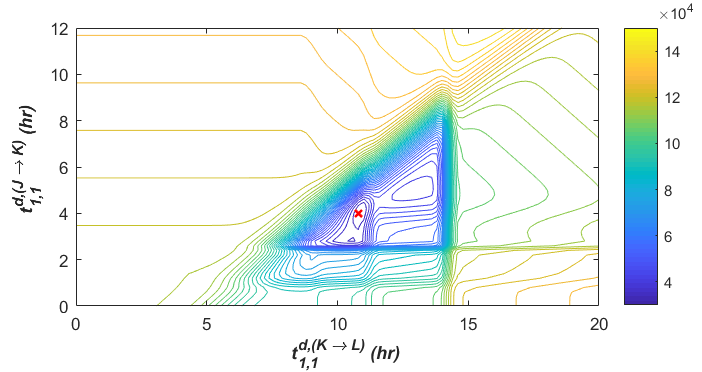}
\caption{Contour plot of objective function ($\kappa$ in $\$$) with respect to two decision variables} 
\label{fig:case1contour}
\end{figure}

\begin{figure}[H]
\setlength{\abovecaptionskip}{3pt}
\setlength{\belowcaptionskip}{5pt}
\centering
\begin{subfigure}[t]{0.4\textwidth}
\includegraphics[width=\textwidth, height=3.5cm]{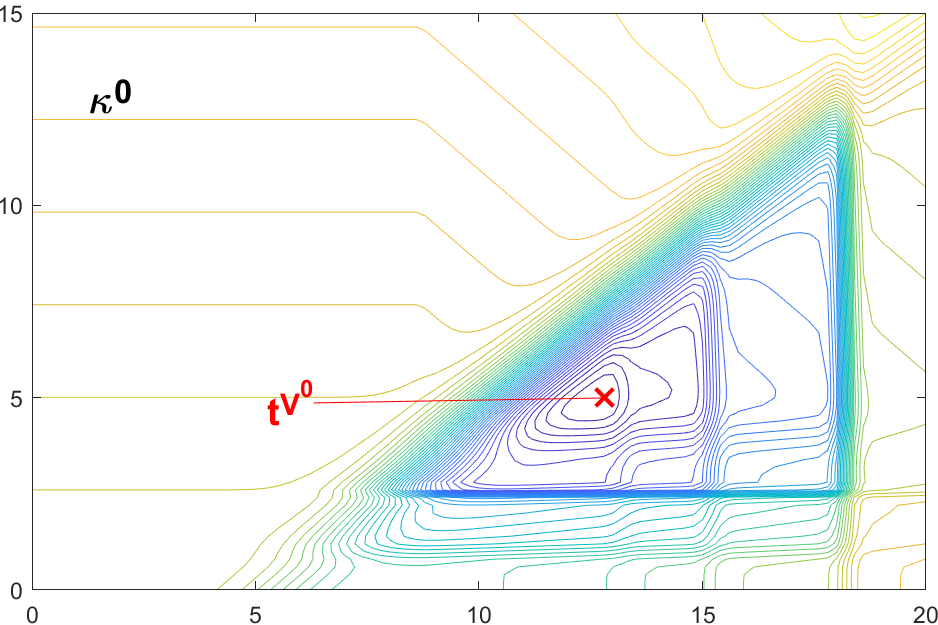}
\caption{}
\end{subfigure}
\begin{subfigure}[t]{0.4\textwidth}
\includegraphics[width=1\linewidth, height=3.5cm]{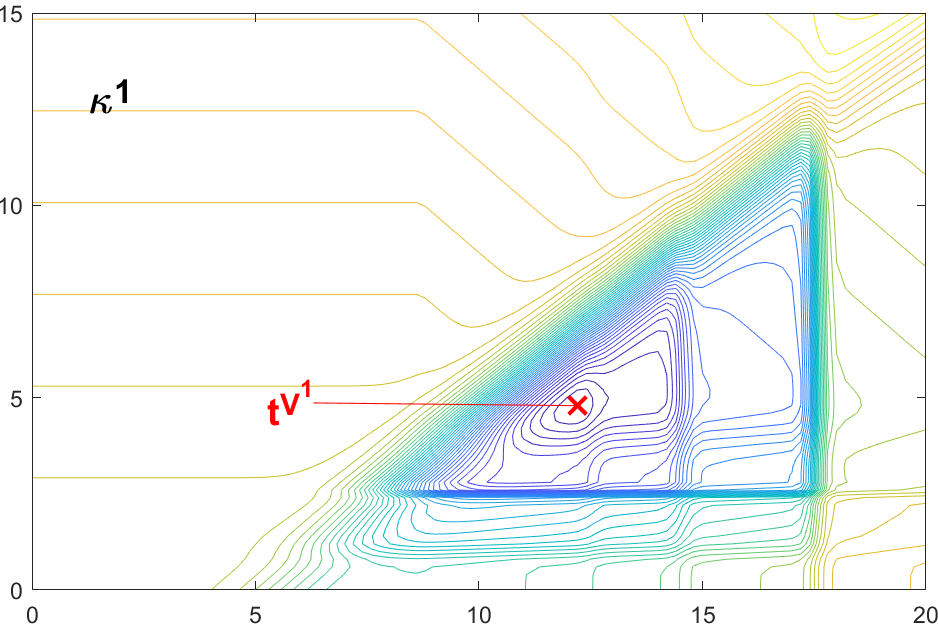}
\caption{}
\end{subfigure}
\begin{subfigure}[b]{0.4\textwidth}
\includegraphics[width=1\linewidth, height=3.5cm]{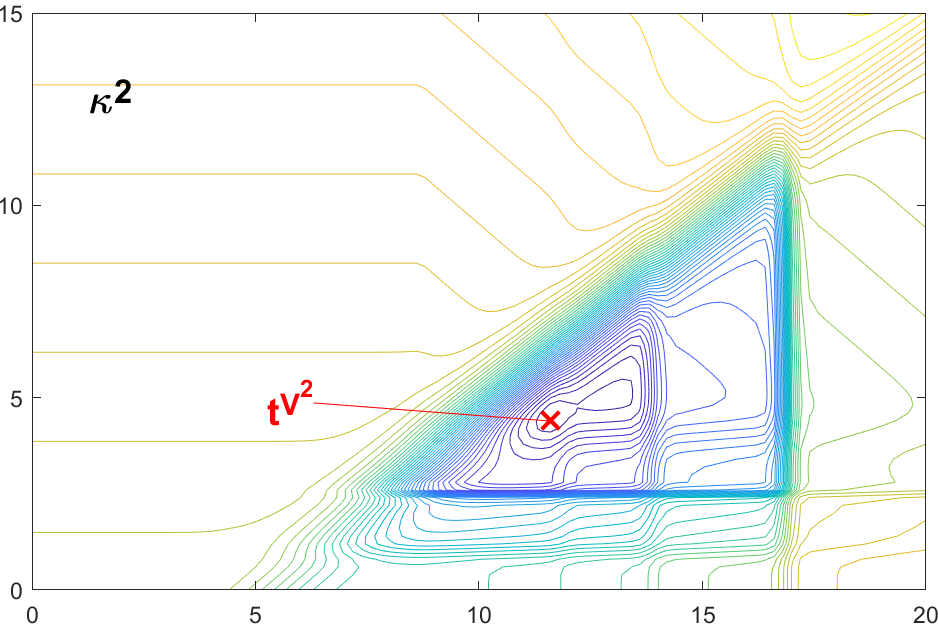}
\caption{}
\end{subfigure}
\begin{subfigure}[b]{0.4\textwidth}
\includegraphics[width=1\linewidth, height=3.5cm]{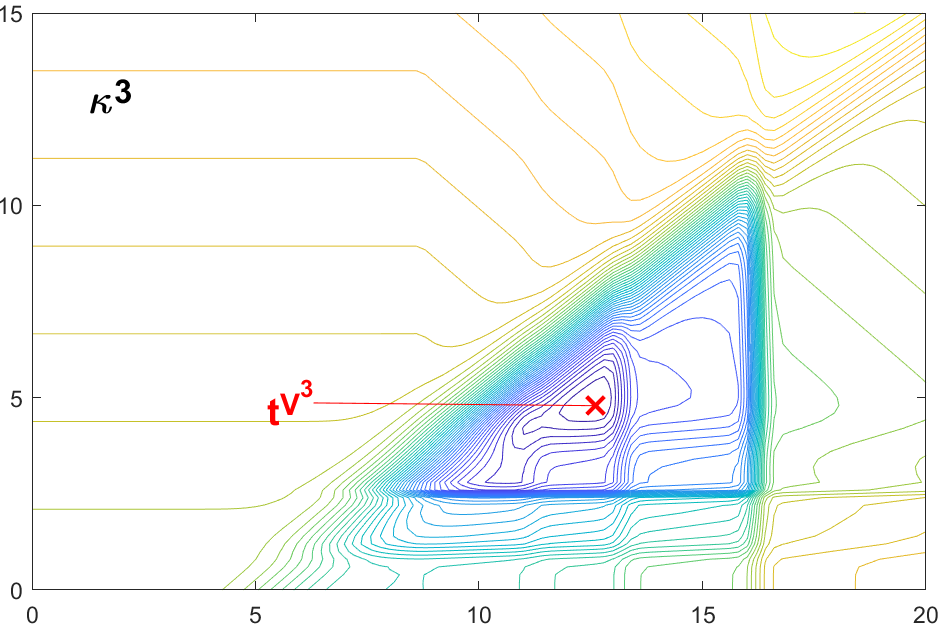} 
\caption{}
\end{subfigure}
\begin{subfigure}[b]{0.4\textwidth}
\includegraphics[width=1\linewidth, height=3.5cm]{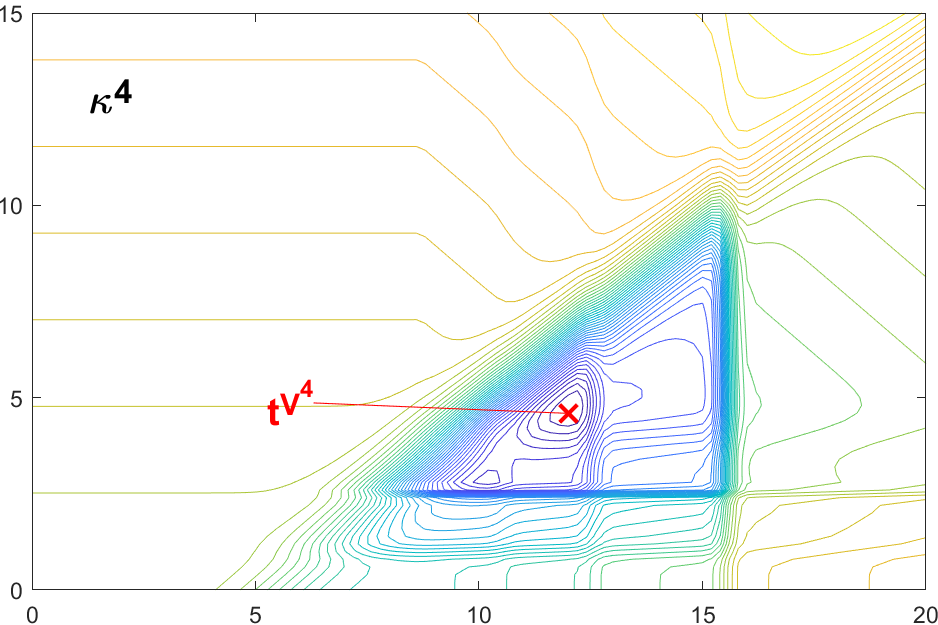}
\caption{}
\end{subfigure}
\begin{subfigure}[b]{0.4\textwidth}
\includegraphics[width=1\linewidth, height=3.5cm]{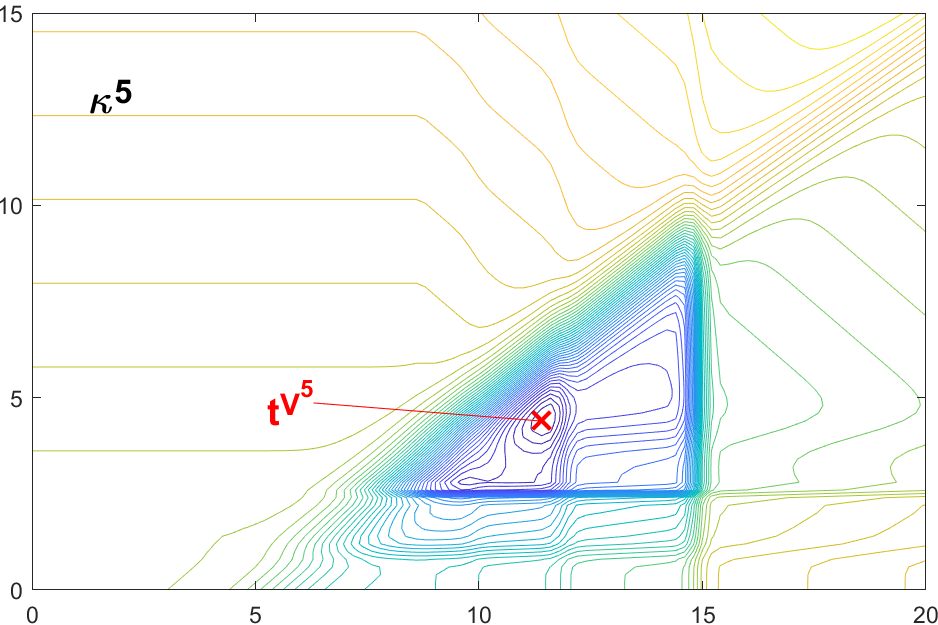}
\caption{}
\end{subfigure}
\caption{Transformation of system in Case 1 using homotopy Algorithm \ref{alg-homo}}
\label{fig:case1homotopy}
\end{figure}

\subsection{ Case 2: A midsize system} 
Case 2 is the network in Figure \ref{fig:network}, with 3 hubs on the first level and 3 hubs on the second level. There are 10 variables, and the resulting minimized cost is $\$54,793$.
We compare our algorithm with NLopt \citep{nlopt}. NLopt is a free/open-source library for nonlinear optimization \achange{which includes several prominent global and local optimization algorithms, such as \cchange{the} DIRECT algorithm \citep{jones1993lipschitzian} and its \cchange{variant} \citep{gablonsky2001locally}, Controlled Random Search algorithms \citep{kaelo2006some}, Multi-Level Single-Linkage algorithm \citep{rinnooy1987stochastic}, Improved Stochastic Ranking Evolution Strategy \citep{runarsson2005search}, Method of Moving Asymptotes \citep{svanberg2002class}, Low-storage BFGS \citep{nocedal1980updating,liu1989limited}, Preconditioned truncated Newton method \citep{dembo1983truncated}, Augmented Lagrangian algorithm \citep{conn1991globally,birgin2008improving}, and others. When using the NLopt, we first use each of their global optimization algorithms separately, and then try to improve each solution by a local solver as suggested by the authors.

We report in Table \ref{table1} the best result that we obtain using all their algorithms.} We keep the number of function evaluations equal for all solvers. Algorithm \ref{alg-homo} is much more effective than NLopt for this problem.
\begin{table}[H]
\centering
\caption{Performance of homotopy algorithm for Case 2}
\begin{tabular}{ | p{4.7cm} | p{2.2cm}| p{3cm} |} 
 \hline
 Solver & Minimized $\kappa (\$)$ & {Number of function evaluations} \\ 
 \hline
 Algorithm 1 & $54,793$ & 319,693 \\ 
 \hline
 Best result obtained by NLopt & $62,979$ & 319,767 \\ 
 \hline
\end{tabular}
\label{table1}
\end{table}


\subsection{ Case 3: Large systems}
We study the performance of Algorithm \ref{alg-homo} with respect to network size, $(n_I, n_J, n_K, n_L)$. Networks of size $(5\omega, \omega, \omega, 5\omega)$ involve $\omega^2 + 5 \omega$ variables.
We define $\eta$ to be the ratio of the best objective function value obtained from Nlopt  to that obtained by our homotopy algorithm, so $\eta > 1$ means that our algorithm delivers a better result.
For each $\omega$ ($\omega = 2,5,10,15,20,25,30,35,40$), we construct 50 problems with random data and report the max, min, and average value of $\eta$. The number of function evaluations is kept equal among all solvers for cases with same $\omega$. Figure \ref{fig:case3} shows the results. We see that Algorithm \ref{alg-homo} outperforms NLopt on all but the smallest problems, and the margin by which it is better increases with the number of variables.

\achange{To elaborate further on the advantage of our homotopy algorithm, consider the case where $\eta =1.1$. In such case, we have decreased the total cost in the system by 10\%, by finding a better minimizer of our cost function, using the homotopy algorithm.}

\achange{The model with the largest number of variables is solved in less than 5 minutes on a 2017 Macbook. Using a more powerful computer, or a distributed system, may deliver the solution in real-time, even for larger systems. }

\begin{figure}[H]
\centering
\includegraphics[width=0.7\textwidth]{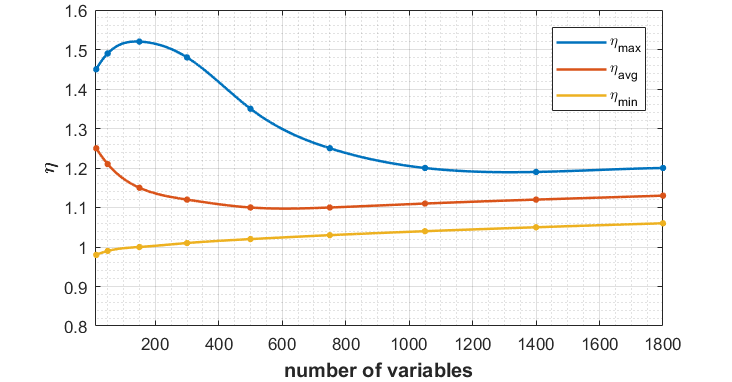}
\caption{Algorithm performance with respect to size of network} 
\label{fig:case3}
\end{figure}


\section{Conclusion} \label{sect-conclusion}
In this paper, we proposed a probabilistic model and a homotopy algorithm for optimizing real-time dispatch decisions in freight transportation systems.
\begin{enumerate}
    \item Our model is probabilistic and hierarchical with two levels of hubs, which has not been previously studied for freight systems. The comprehensive cost function is based on the expected cost of real-time dispatch decisions.
     \item We \cchange{are able to allow a nonlinear penalty for late delivery of cargo, not previously considered in the literature for hierarchical networks.}
    \item Our new homotopy algorithm (empirically) finds lower cost solutions than standard optimization algorithms on this problem by taking advantage of its structure. \cchange{We believe that although customized  nonlinear optimization algorithms are not common in the freight transportation literature, they have great potential, especially in real-time applications.}
         \item \cchange{Our} algorithm is computationally efficient and parallelizable.
\end{enumerate}

\section*{Acknowledgements}
    The authors thank Thomas Goldstein for helpful discussions.

\section*{References}

\bibliography{references}

\end{document}

%% file: commondefinitions.tex
\definecolor{plum}{rgb}{0.45,0,.66}
\definecolor{mythistle}{rgb}{.99,.195,.133}
\definecolor{myred}{cmyk}{0.000000,1.000000,1.000000,0.1}
\definecolor{myblue}{cmyk}{1.000000,0.750000,0.000000,0.1}
\definecolor{mybgn}{cmyk}{0.850000,0.350000,0.000000,0.1}
\definecolor{mygrn}{cmyk}{0.750000,0.000000,1.000000,0.2}

\newcommand{\bl}{\begin{itemize}}
\newcommand{\el}{\end{itemize}}
\newcommand{\be}{\begin{enumerate}}
\newcommand{\ee}{\end{enumerate}}

\newcommand{\bea}{\begin{eqnarray*}}
\newcommand{\eea}{\end{eqnarray*}}

\newcommand{\beq}{\begin{equation}}
\newcommand{\eeq}{\end{equation}}

\newcommand{\bmx}{\left[ \begin{array}}
\newcommand{\emx}{\end{array} \right]}

